\documentclass[oneside,12pt]{amsart}

\usepackage{amscd,amsmath,latexsym,bm,pdfpages,verbatim,amsthm}
\usepackage[mathcal]{euscript}
\usepackage{graphicx, mathtools, hyperref,mathtools,enumerate}
\usepackage{amssymb}          
\usepackage{epstopdf}
\usepackage[all]{xy}
\usepackage{tikz}
\usepackage{tikz-cd}
\usepackage{mathtools}

\newtheorem{thm}{Theorem}[section]
\newtheorem{cor}[thm]{Corollary}
\newtheorem{lem}[thm]{Lemma}

\theoremstyle{definition}
\newtheorem{defin}[thm]{Definition}
\theoremstyle{definition}

\theoremstyle{definition}
\newtheorem{exm}[thm]{Example}

\newtheorem{remark}[thm]{Remark}

\theoremstyle{remark}
\newtheorem*{rem}{Remark}

\DeclareMathOperator*{\moplus}{\text{\raisebox{0.28ex}{\scalebox{0.75}{$\bigoplus$}}}}
\DeclareMathOperator*{\mvee}{\text{\raisebox{0.28ex}{\scalebox{0.7}{$\bigvee$}}}}

\DeclareMathOperator*{\mwedge}{\text{\raisebox{0.28ex}{\scalebox{0.7}{$\bigwedge$}}}}
\DeclareMathOperator*{\mprod}{\text{\raisebox{0.28ex}{\scalebox{0.67}{$\prod$}}}}
\DeclareMathOperator*{\lprod}{\text{\raisebox{0.28ex}{\scalebox{0.76}{$\prod$}}}}
\DeclareMathOperator*{\mcup}{\text{\raisebox{0.28ex}{\scalebox{0.78}{$\bigcup$}}}}

\begin{document}
\def\X#1#2{r(v^{#2}\ds{\prod_{i \in #1}}{x_{i}})}
\def\skp#1{\vskip#1cm\relax}
\def\mr#1{\mathring{#1}}
\def\nm#1{\mbox{{\normalsize $#1$}}}
\def\lm#1{\mbox{{\large $#1$}}}
\def\block{\rule{2.4mm}{2.4mm}}
\def\nd{\noindent}
\def\Z{\mathbb{Z}}
\def\becomes{\colon\hspace{-2,5mm}=}
\def\ds{\displaystyle}
\definecolor{teagreen}{rgb}{0.82, 0.94, 0.75}
\definecolor{cambridgeblue}{rgb}{0.64, 0.76, 0.68}
\definecolor{cream}{rgb}{1.0, 0.99, 0.82}
\definecolor{ivory}{rgb}{1.0, 1.0, 0.94}
\def\red{\color{red}}
\def\blue{\color{blue}}
\def\black{\color{black}}
\def\s{\sigma}
\def\xa{(\underline{X},\underline{A})}
\numberwithin{equation}{section}
\title[Symmetric and polyhedral products]{Symmetric Products and a Cartan-type formula for polyhedral products}


\skp{0.2}

\author[A.~Bahri]{A.~Bahri}
\address{Department of Mathematics,
Rider University, Lawrenceville, NJ 08648, U.S.A.}
\email{bahri@rider.edu}

\author[M.~Bendersky]{M.~Bendersky}
\address{Department of Mathematics
CUNY,  East 695 Park Avenue New York, NY 10065, U.S.A.}
\email{mbenders@hunter.cuny.edu}

\author[F.~R.~Cohen]{F.~R.~Cohen}
\address{}

\author[S.~Gitler]{S.~Gitler}
\address{}



\begin{abstract}
We give a geometric method for determining the cohomology groups  of
a polyhedral product $Z\big(K;(\underline{X}, \underline{A})\big)$, under suitable 
freeness conditions  or with coefficients taken in a  field $k$. This is done by considering first the special case 
where {  the pair} $(X_{i},A_{i }) = (B_{i}\vee C_{i}, B_{i}\vee E_{i})$ for all $i$, and $E_{i}\hookrightarrow C_{i} $ is a 
null homotopic inclusion. We derive a decomposition for these polyhedral products which resembles 
a Cartan  formula. The theory of symmetric products is used then  to generalize the  result 
 to arbitrary  
polyhedral products $Z\big(K;(\underline{X}, \underline{A})\big)$.
This leads to a direct computation of the Hilbert-Poincar\'e series for $Z\big(K;(\underline{X}, \underline{A})\big)$
and to other applications. 
\end{abstract}
\maketitle
\tableofcontents
\section{Introduction}\label{sec:introduction}
Our purpose is to  recall some standard properties of infinite symmetric products, known also as the 
Dold-Thom construction \cite{doldthom}, and to develop some related maps which are defined for polyhedral products. 
The main feature is that topological maps on the level of infinite symmetric products applied to polyhedral products 
can be defined directly from homological information. 

Polyhedral products $Z\big(K;(\underline{X}, \underline{A})\big)$, \cite{bbcg1},  are defined
for a simplicial complex $K$ on the vertex set $[m] = \{1,2,\ldots,m\}$, and a family of pointed CW pairs
$$(\underline{X}, \underline{A}) = \big\{(X_{i},A_{i}): i =1,2,\ldots,m\big\}.$$

{\let\thefootnote\relax\footnote{\hspace{-4.5mm}This work was supported in part by grant 426160 from Simons Foundation. 
The authors are grateful to the Fields Institute for a conducive environment during the 
Thematic Program: Toric Topology and Polyhedral Products. \mbox{The first author acknowledges Rider 
University for a Spring 2020 research leave.}}}

\nd They are natural subspaces of the Cartesian product $X_{1}\times X_{2}\times\cdots\times X_{m}$, in
such a way that if $K = \Delta^{m-1}$, the $(m-1)$-simplex, then 
$$Z\big(K;(\underline{X}, \underline{A})\big) = X_{1}\times X_{2}\times\cdots\times X_{m}.$$
More specifically, we consider $K$ to be a category
where the objects are the simplices of $K$ and the morphisms $d_{\sigma,\tau}$ are the inclusions 
$\sigma \subset \tau$.  A  polyhedral product is given as the colimit of a diagram 
$D_{(\underline{X}, \underline{A})}: K \to CW_{\ast}$, where at each $\sigma \in K$, we set 
\begin{equation}\label{eqn:d.sigma}
D_{(\underline{X}, \underline{A})}(\sigma) =\lprod^m_{i=1}W_i,\quad {\rm where}\quad
W_i=\left\{\begin{array}{lcl}
X_i  &{\rm if} & i\in \sigma\\
A_i &{\rm if} & i\in [m]-\sigma.
\end{array}\right.
\end{equation}

\nd  Here, the colimit is a union given by 
$$Z(K; (\underline{X}, \underline{A})) = \mcup_{\sigma \in K}D_{(\underline{X}, \underline{A})}(\sigma),$$
 but the full colimit structure is used heavily in  the development of the elementary theory.
Notice that when $\sigma \subset \tau$  then 
$D_{(\underline{X}, \underline{A})}(\sigma) \subseteq D_{(\underline{X}, \underline{A})}(\tau)$. 
In the case that $K$ itself is a simplex, 
$$Z(K; (\underline{X}, \underline{A}))\; =\; \lprod\limits_{i=1}^{m}{X_i}.$$ 

Polyhedral products were formulated first for the case $(X_i,A_i) = (D^2,S^1)$  by V.~Buchstaber
and T.~Panov in \cite{bp3}; they called their spaces {\em moment-angle complexes\/}.

\skp{0.2}
In a way entirely similar to that above, a related space $\widehat{Z}(K; (\underline{X}, \underline{A}))$, called the 
{\em polyhedral  smash product\/},  is defined by replacing the Cartesian product everywhere above by the 
smash product. That is,
$$\widehat{D}_{(\underline{X}, \underline{A})}(\sigma) =\mwedge ^m_{i=1}W_i \quad {\rm and} \quad
\widehat{Z}(K; (\underline{X}, \underline{A})) = 
\mcup_{\sigma \in K}\widehat{D}_{(\underline{X}, \underline{A})}(\sigma)$$
\nd  with
$$\widehat{Z}(K; (\underline{X}, \underline{A})) \; \subseteq \; \mwedge\limits_{i=1}^{m}{X_i}.$$
The polyhedral smash product is related to the polyhedral product by the stable decomposition 
discussed  in \cite{bbcg1} and \cite{bbcg2}. We denote by $(\underline{X}, \underline{A})_J$
the restricted family of CW-pairs $\big\{(X_j,A_j\big)\}_{j\in J}$, and by $K_J$, the full subcomplex
on $J \subset [m]$. 
\begin{thm}\cite[Theorem 2.10]{bbcg2}\label{thm:bbcgsplitting}
Let $K$ be an abstract simplicial complex  on vertices $[m]$.  Given a family $\{(X_j , A_j)\}_{j=1}^m$ of
pointed  pairs  of CW-complexes, there is a natural pointed homotopy equivalence 
\begin{equation}\label{eqn:splitting}
H\colon \Sigma{\big(Z\big(K;(\underline{X},\underline{A})\big)\big)} \longrightarrow 
\Sigma\big(\mvee_{J\subseteq [m]}\widehat{Z}\big(K_J; (\underline{X}, \underline{A})_J\big)\big).
\end{equation}
\end{thm}

In many of the most important cases, the spaces $\widehat{Z}\big(K_J; (\underline{X}, \underline{A})_J\big)\big)$
can be identified explicitly, \cite{bbcg2}.
Aside from the various unstable and stable splitting theorems, \cite{bbcg1,gt,gtshifted,ikshifted,ik5},
there is an extensive history of  computations of the cohomology 
groups and rings of various families of polyhedral products,  
\cite[Sections 5, 8 and 11]{bbcgsurvey}, see also \cite{ldm1,franz,bbp,zheng,zheng2,bbcg10,cai1,caichoi}. 

Some very early calculations of the cohomology of certain moment-angle complexes, (the case 
$(X_i,A_i) = (D^2,S^1)$ for all $i = 1,2,\ldots,m$), appeared in the work of Santiago L\'opez de Medrano 
\cite{ldm1}, though at that time the spaces he studied were
not recognized to have the structure of a moment-angle complex. 
The cohomology algebras of all moment-angle complexes was computed first by M.~Franz \cite{franz} and by
 I.~Baskakov, V.~Buchstaber and T.~Panov in \cite{bbp}.

The cohomology of the polyhedral product  $Z\big(K;  (\underline{X}, \underline{A}) \big)$,  for 
$(\underline{X},\underline{A})$, satisfying certain freeness conditions, (coefficients in a field $k$ for example),  
was computed using 
a spectral sequence  by the authors in \cite{bbcg10}. A computation using different methods by Q.~Zheng can 
be found in \cite{zheng, zheng2}. 

The special family of CW pairs $(\underline{U}, \underline{V}) = (\underline{B\vee C}, \underline{B\vee E})$
satisfying  the condition  that  for all $i$, $(U_{i},V_{i }) = (B_{i}\vee C_{i}, B_{i}\vee E_{i})$, where 
$E_{i}\hookrightarrow C_{i} $ is a
null homotopic inclusion, is called {\em wedge decomposable\/}.
As announced in \cite[Section 12]{bbcgsurvey},  one goal of the current paper is to show that for wedge 
decomposable pairs 
$ (\underline{U}, \underline{V})$,  the algebraic decomposition given by 
the spectral sequence calculation \cite[Theorem $5.4$]{bbcg10}  is a consequence of an underlying geometric splitting. 
{\em Moreover, the results of this observation extend to general based CW-pairs of finite type.}

This paper is partly a revised version of the authors' unpublished preprint from 2014, which in turn originated from an 
earlier preprint from 2010.   In addition, the results of this paper  have been extended to describe the product 
structure in the cohomology and these will appear separately.
 
We begin in Section \ref{sec:wdpairs} by deriving for wedge decomposable pairs
$(\underline{U},\underline{V})$
an explicit decomposition of the polyhedral product into a wedge of much simpler spaces, (Theorem \ref{thm:main}
and Corollary \ref{cor:wedge}). In particular, this allows us to identify explicit additive generators for 
$H^{\ast}\big( Z\big(K;  (\underline{U}, \underline{V})\big)$.  The proof in Section
\ref{sec:thmmain} is an induction based on a filtration of the polyhedral product which is introduced in Section 
\ref{sec:filtration}. 

These decompositions give a direct framework for deducing, (Theorem \ref{thm:Cartan_and_general_homology}),
an analogous homological Cartan formula for 
the additive structure of the homology of $Z\big(K;(\underline{X},\underline{A})\big)$ for {\em any\/} family of pairs of finite, 
pointed, path-connected CW-complexes $(\underline{X},\underline{A})$, This is done 
by applying properties of the infinite symmetric product $SP(-)$ and
the polyhedral product. Namely, given pointed pairs of finite, path-connected CW-complexes, 
$(\underline{X},\underline{A})$, there exist 
pointed pairs of path-connected CW-complexes 
$$(\underline{U}, \underline{V}) = (\underline{B\vee C}, \underline{B\vee E})$$ 
together with a homotopy equivalence
$$SP\big(\widehat{Z}(K;(\underline{U},\underline{V}))\big) \longrightarrow
SP\big(\widehat{Z}(K;(\underline{X},\underline{A}))\big)$$
\nd Applications of the additive results  comprise Sections \ref{sec:hpseries} and \ref{sec:applications}.

\section{The polyhedral product of wedge decomposable pairs}\label{sec:wdpairs}
We begin with a definition.

\begin{defin}\label{def:wedgedecomp}
The special family of CW pairs $(\underline{U}, \underline{V}) = (\underline{B\vee C}, \underline{B\vee E})$
satisfying   $(U_{i},V_{i }) = (B_{i}\vee C_{i}, B_{i}\vee E_{i})$ for all $i$, where $E_{i}\hookrightarrow C_{i} $ is a
null homotopic inclusion, is called {\em wedge decomposable\/}.
\end{defin} 
The fact that the smash product distributes over wedges of spaces, leads to the characterization of the smash polyhedral 
product in a way which resembles a {\em Cartan formula\/}.

\begin{thm} \label{thm:main} (Cartan Formula)
Let $(\underline{U}, \underline{V}) = (\underline{B\vee C}, \underline{B\vee E})$ be a wedge decomposable
pair, then there is a homotopy equivalence
$$\widehat{Z}\big(K;(\underline{U}, \underline{V})\big) \longrightarrow 
\mvee_{I\leq [m]}\Big(\widehat{Z}\big(K_{I};(\underline{C},\underline{E})_I\big) \wedge
\widehat{Z}\big(K_{[m]- I};(\underline{B},\underline{B})_{[m]-I}\big)\Big)$$

\nd which is natural with respect to maps of decomposable pairs.  Of course, 
$$\widehat{Z}\big(K_{[m]- I};(\underline{B},\underline{B})_{[m]-I}\big) \;=\;  \mwedge_{j\; \in\; [m]-I}B_j$$
with the convention that 
$$\widehat{Z}\big(K_{\varnothing};(\underline{B},\underline{B})_{\varnothing}\big),\;
\widehat{Z}\big(K_{\varnothing};(\underline{C},\underline{E})_{\varnothing}\big)
 \;\; {\rm{and}}\;\; \widehat{Z}\big(K_{I};(\varnothing,\varnothing)_{I}\big)
 \;=\; S^0.$$
\end{thm}
\skp{0.3}
We can decompose $\widehat{Z}\big(K;(\underline{U}, \underline{V})\big)$ further by applying
(a generalization of) the Wedge Lemma. We recall first the definition of a link.
\begin{defin}\label{defn:link}
For $\sigma$  a simplex in a simplicial complex $\mathcal{K}$, $\text{lk}_{\sigma}(\mathcal{K})$  
{\em the link of\/} $\sigma$ {\em in\/} $\mathcal{K}$, is defined to be the simplicial complex  for which
$$ \tau \in \text{lk}_{\sigma}(K)\quad \text{if and only if}\quad \tau \cup \sigma \in \mathcal{K}.$$
\end{defin} 
\begin{thm}\label{thm:wlemma} \cite[Theorem $2.12$]{bbcg1}, \cite[Lemma 1.8]{zz}
Let $\mathcal{K}$ be a simplicial complex on $[m]$ and $(\underline{C}, \underline{E})$ a family of CW pairs
satisfying $E_i \hookrightarrow C_i$ is null homotopic for all $i$ then
$$\widehat{Z}(\mathcal{K}; (\underline{C}, \underline{E})) \simeq \mvee_{\sigma \in \mathcal{K}}
|\Delta(\overline{\mathcal{K}})_{<\sigma}|\ast
\widehat{D}_{\underline{C},\underline{E}}^{[m]}(\sigma)$$
\nd where $|\Delta(\overline{\mathcal{K}})_{<\sigma}| \cong |\text{lk}_{\sigma}(\mathcal{K})|$, the 
realization of the link of $\sigma$ in the
simplicial complex $\mathcal{K}$  and  
\begin{equation}\label{eqn:ced.sigma}
\widehat{D}_{\underline{C},\underline{E}}^{[m]}(\sigma)  =\mwedge^{m}_{j=1}W_{i_{j}},\quad {\rm with}\quad
W_{i_{j}}=\left\{\begin{array}{lcl}
C_{i_{j}}  &{\rm if} & i_{j}\in \sigma\\
E_{i_{j}}  &{\rm if} & i_{j}\in [m]-\sigma.
\end{array}\right.
\end{equation} \hfill $\square$
\end{thm}
Applying this to the decomposition of Theorem \ref{thm:main}, we get

\begin{cor}\label{cor:wedge} There is a homotopy equivalence
$$\widehat{Z}\big(K;(\underline{U}, \underline{V})\big) \longrightarrow 
\mvee_{I\leq [m]}\Big(\big(\mvee_{\sigma \in K_{I}} |lk_{\sigma}(K_{I})|\ast
\widehat{D}_{\underline{C},\underline{E}}^{I}(\sigma)\big) \wedge
\widehat{Z}\big(K_{[m]- I};(\underline{B},\underline{B})_{[m]-I}\big)\Big).$$

\nd where $\widehat{D}_{\underline{C},\underline{E}}^{I}(\sigma)$ is as in \eqref{eqn:ced.sigma} with
$I$ replacing $[m]$.
\end{cor}

\nd Combined with Theorem \ref{thm:bbcgsplitting}, this gives a complete
description of the topological spaces $Z\big(K;(\underline{U}, \underline{V})\big)$ for wedge decomposable pairs
$(\underline{U}, \underline{V})\big)$. 
\skp{0.1}
The case $E_i \simeq \ast$ simplifies further by \cite[Theorem $2.15$]{bbcg2} to give the next corollary.
\begin{cor}\label{cor:E.a.point}
For wedge decomposable pairs of the form $(\underline{B\vee C}, \underline{B})$, corresponding to $E_i \simeq \ast$
for all $i = 1,2,\ldots,m$, there are homotopy equivalences
$$\widehat{Z}\big(K_{I};(\underline{C},\underline{E})_I\big) \simeq \widehat{Z}\big(K_{I};(\underline{C},\ast)_I\big)
\simeq  \widehat{C}^I,$$
and so Theorem \ref{thm:main} gives $\widehat{Z}\big(K;(\underline{B\vee C}, \underline{B})\big) \simeq
\ds{\mvee_{I\leq [m]}}\big(\widehat{C}^I\wedge \widehat{B}^{([m]-I)}\big)$. \hfill $\square$
\end{cor}

\nd Notice here that the Poincar\'e series for the space $\widehat{Z}\big(K;(\underline{B\vee C}, \underline{B})\big)$
follows easily from Corollary \ref{cor:E.a.point}. 

\begin{rem} In comparing these observations with  \cite[Theorem $5.4$]{bbcg10}, notice that the links appear in the terms
$\widehat{Z}\big(K_{I};(\underline{C},\underline{E})_I\big)$. Also, while  Theorem \ref{thm:main} and Corollary \ref{cor:wedge} 
give a geometric underpinning for the cohomology calculation in \cite[Theorem $5.4$]{bbcg10} for wedge decomposable pairs, 
the geometric splitting does not require that $E,B$ or $C$ have torsion-free cohomology
\end{rem}

\section{A filtration}\label{sec:filtration}
We begin by reviewing the filtration on polyhedral products used for the spectral sequence calculation in \cite{bbcg10},
 following  \cite[Section 2]{bbcg10}, where more details can be found.  The  length-lexicographical 
ordering on the faces  of the $(m-1)$-simplex $\Delta[m-1]$ is induced by an ordering on the vertices.
This is the left lexicographical ordering on strings of varying lengths with shorter strings taking
precedence.
The ordering gives a filtration on $\Delta[m-1]$ by $$F_{t}(\Delta[m-1]) = \mcup_{s\leq t}\sigma_s.$$

\nd In turn, this gives a total ordering on the simplices of a simplicial $K$ on $m$ vertices
\begin{equation}\label{eqn:simplexordering}
\sigma_{0} = \varnothing< \sigma_{1}<  \sigma_{2}<\ldots <\sigma_{t}<\ldots< \sigma_{s}
\end{equation}
\nd via the natural inclusion $$K\; \subset \; \Delta[m-1].$$
This is filtration preserving in the sense that $F_{t}K = K \cap F_{t}\Delta[m-1]$. 
\begin{exm}
Consider $[m] = [3]$ and 
$$K  = \big\{\phi, \{v_1\},\{v_2\},\{v_3\}, \big\{\{v_1\},\{v_3\}\big\}, \big\{\{v_2\},\{v_3\}\big\}\big\}$$ 
with the realization consisting of two edges with a common vertex.
Here the  length-lexicographical ordering on the two-simplex $\Delta[2]$ is 
$$\phi <  v_1<v_2<v_3<v_1v_2<v_1v_3<v_2v_3<v_1v_2v_3$$
and so the induced ordering on $K$ is
$$\phi <  v_1<v_2<v_3<v_1v_3<v_2v_3\/.$$
\end{exm}
\begin{rem}
 Notice that if $t<m$, then $F_{t}K$ will contain  {\em ghost\/} vertices, that is, vertices which are in $[m]$ but are not 
considered simplices, They do however label Cartesian product factors in the polyhedral product. 
\end{rem}
As described in \cite[Section 2]{bbcg10},
this induces a natural filtration on the polyhedral product $Z\big(K;(\underline{X}, \underline{A})\big)$ and the smash 
polyhedral product  $\widehat{Z}\big(K;(\underline{X}, \underline{A})\big)$ as follows:
$$F_{t}{Z}(K;(\underline{X},\underline{A})) \;=\; \mcup_{k \leq t}D_{(\underline{X}, \underline{A})}({\sigma_k})
\quad {\rm{and}} \quad F_{t}\widehat{Z}(K;(\underline{X},\underline{A})) = \mcup_{k \leq t}
\widehat{D}_{\underline{X},\underline{A}}(\sigma_{k}).$$

\nd Notice also that the filtration satisfies
\begin{equation}\label{eqn:flitration}
F_{t}\widehat{Z}(K;(\underline{X},\underline{A})) \;=\; \widehat{Z}(F_{t}K;(\underline{X},\underline{A})).
\end{equation}
\skp{0.3}

\section{The proof of Theorem \ref{thm:main}}\label{sec:thmmain}
Let the family of CW pairs $(\underline{U}, \underline{V})$ be wedge decomposable as in 
Definition \ref{def:wedgedecomp}.
 We begin by checking that Theorem \ref{thm:main} holds for $F_{0}\widehat{Z}(K;(\underline{U},\underline{V}))$.
In this case $F_{0}K$ consists of the empty simplex, (the boundary of a point), and $m-1$ ghost vertices.
So, 
\begin{align}\label{eqn:F0}
\widehat{Z}(F_{0}K;(\underline{U},\underline{V})) \;=\; V_1 \wedge V_2 \cdots \wedge V_m
\;=\;(B_1 \vee E_1) \wedge  (B_2 \vee E_2)  \wedge \cdots \wedge (B_m \vee E_m).
\end{align}
Next, fix $I = (i_1,i_2,\ldots, i_k) \subset [m]$ and set $[m]-I = (j_i,j_2,\ldots j_{m-k})$. Then
$$\widehat{Z}\big(F_{0}K_{I};(\underline{C},\underline{E})_I\big)\wedge 
\widehat{Z}\big(K_{[m]- I};(\underline{B},\underline{B})_{[m]-I}\big)= (E_{i_1}\wedge E_{i_2}\wedge \cdots E_{i_k})\wedge
(B_{j_1}\wedge B_{j_2}\wedge \cdots B_{j_{m-k}}).$$
    \nd is the $I$-th wedge term in the expansion of the right hand side of \eqref{eqn:F0}. This confirms Theorem \ref{thm:main}
for $t=1$. 
\skp{0.3}
We suppose next the induction hypothesis that 
$$F_{t-1}\widehat{Z}(K;(\underline{U},\underline{V})) \;\simeq\;
\mvee_{I\leq [m]}\widehat{Z}\big(F_{t-1}K_{I};(\underline{C},\underline{E})_I\big)\wedge 
\widehat{Z}\big(K_{[m]- I};(\underline{B},\underline{B})_{[m]-I}\big),$$
\nd with a view to verifying it for $F_{t}$.  The definition of the filtration gives
{\fontsize{11}{12}\selectfont
\begin{align}\label{eqn:union}
F_{t}\widehat{Z}(K;(\underline{U},\underline{V})) \;&=\; 
\widehat{D}_{\underline{U},\underline{V}}(\sigma_{t}) \;\cup\; 
F_{t-1}\widehat{Z}(K;(\underline{U},\underline{V}))  \\
\nonumber &\simeq\;
\widehat{D}_{\underline{U},\underline{V}}(\sigma_{t}) \;\cup\; 
\mvee_{I\leq [m]}\widehat{Z}\big(F_{t-1}K_{I};(\underline{C},\underline{E})_I\big)\wedge 
\widehat{Z}\big(K_{[m]- I};(\underline{B},\underline{B})_{[m]-I}\big).
\end{align}}
\nd The space \;$\ds{\widehat{D}_{\underline{U},\underline{V}}(\sigma_{t})}$\; is the smash product
\begin{equation}\label{eqn:dsigma}
\mwedge^{m}_{j=1}B_{j}\vee Y_{j},\quad {\rm with}\quad
Y_{j}=\left\{\begin{array}{lcl}
C_{j}  &{\rm if} & j\in \sigma_{t}\\
E_{j}  &{\rm if} & j\notin \sigma_{t}.
\end{array}\right.
\end{equation}
\vspace{0.5mm}

\nd After a shuffle of wedge factors, the space
 $\ds{\widehat{D}_{\underline{U},\underline{V}}(\sigma_{t})}$ becomes
\begin{multline}\label{eqn:induction}
\mvee_{I\leq [m],\; \sigma_{t}\in I}
\widehat{D}_{\underline{C},\underline{E}}^{I}(\sigma_{t}) \wedge 
\widehat{Z}\big(K_{[m]- I};(\underline{B},\underline{B})_{[m]-I}\big)\;
\; \vee \\
\mvee_{I\leq [m],\; \sigma_{t}\notin I}
\widehat{Z}\big(K_{I};(\underline{C},\underline{E})_I\big) \wedge 
\widehat{Z}\big(K_{[m]- I};(\underline{B},\underline{B})_{[m]-I}\big)
\end{multline} 
\skp{0.2}
\nd where the space $\widehat{D}_{\underline{C},\underline{E}}^{I}(\sigma_{t})$ is defined by
\eqref{eqn:ced.sigma}. 
\skp{0.6}
\begin{rem}
Notice here the relevant fact that the number of subsets $I \leq [m]$ is
the same as the number of wedge summands in the expansion of \eqref{eqn:dsigma}, namely
$2^{m}$.
\end{rem}

\nd The right-hand wedge summand in \eqref{eqn:induction} is a subset of
$$\mvee_{I\leq [m]}\widehat{Z}\big(F_{t-1}K_{I};(\underline{C},\underline{E})_I\big)\wedge 
\widehat{Z}\big(K_{[m]- I};(\underline{B},\underline{B})_{[m]-I}\big)$$ 
\nd and so,
{\fontsize{9.7}{12}\selectfont
\begin{align*}
\mvee_{I\leq [m],\; \sigma_{t}\notin I}
\hspace{-2mm}\widehat{Z}\big(K_{I};(\underline{C},\underline{E})_I\big) \wedge 
\widehat{Z}\big(K_{[m]- I};(\underline{B},\underline{B})_{[m]-I}\big)\;\; &\mcup 
\mvee_{I\leq [m]}\widehat{Z}\big(F_{t-1}K_{I};(\underline{C},\underline{E})_I\big)\wedge 
\widehat{Z}\big(K_{[m]- I};(\underline{B},\underline{B})_{[m]-I}\big)\\[1mm]
&= \mvee_{I\leq [m]}\widehat{Z}\big(F_{t-1}K_{I};(\underline{C},\underline{E})_I\big)\wedge 
\widehat{Z}\big(K_{[m]- I};(\underline{B},\underline{B})_{[m]-I}\big).
\end{align*}}
Finally, for each 
${I\leq [m]}$ with ${\sigma_{t}\in I}$,  we have   

\begin{equation}
\widehat{D}_{\underline{C},\underline{E}}^{I}(\sigma_{t}) \cup 
\widehat{Z}\big(F_{t-1}K_{I};(\underline{C},\underline{E})_I\big) \;=\; 
\widehat{Z}\big(F_{t}K_{I};(\underline{C},\underline{E})_I\big).
\end{equation}
\skp{0.05}
\nd This concludes the inductive step to give
\begin{equation}\label{eqn:ftuv}
F_{t}\widehat{Z}(K;(\underline{U},\underline{V})) \;\simeq\;
\mvee_{I\leq [m]}\widehat{Z}\big(F_{t}K_{I};(\underline{C},\underline{E})_I\big)\wedge 
\widehat{Z}\big(K_{[m]- I};(\underline{B},\underline{B})_{[m]-I}\big).
\end{equation}
\nd It is straightforward to explicitly check the steps above in the case of $F_{0}$ and
$F_{1}$. This completes the proof. \hfill $\Box$

\section{Symmetric products}\label{sec:sym.prod}
We begin with a definition. 
\begin{defin} Let $(X,\ast)$ denote a pointed topological space. The $m$-{\em fold symmetric product\/} for $(X,\ast)$ 
is  the orbit space
$$SP^m(X) = X^m\big/\Sigma_m$$
where the symmetric group on $m$-letters \/$\Sigma_m$ acts on the left by permutation of coordinates.
There are natural maps
\begin{equation}\label{eqn:mape}
\begin{array}{lccc}
\bm{e}\colon& SP^m(X)& \longrightarrow &SP^{m+1}(X)\\[1.9mm]
     & [x_1, x_2, \ldots, x_m]  & \mapsto     &[x_1, x_2, \ldots, x_m,\ast]
\end{array}
\end{equation}

\nd which allow for the definition of the {\em infinite symmetric product\/} as a colimit 
$$SP (X) = \ds{{\rm co}\hspace{-0.6mm}\lim_{\hspace{-4mm}1\leq m}}\;SP^{m}(X).$$
The colimit  becomes a filtered unital commutative monoid under concatenation, with $\ast$ as
the unit. Furthermore, there is a natural inclusion
\begin{equation}\label{eqn:defofE}     
\begin{array}{lccc}
E_X\colon\!\!&  X & \longrightarrow &  SP(X)\\[0.1mm]
                     &p         & \mapsto     &[p]
\end{array}.
\end{equation}
\end{defin}

One version of the classical Dold-Thom theorem is as follows.
\begin{thm}\label{thm:doldthom}\cite{doldthom}
Given a pointed, path-connected pair of finite CW-complexes $(X,A,\ast)$
(where $A$ is a closed subcomplex of $X$), then the following hold.
\begin{enumerate}
\item $SP(X)$ is homotopy equivalent to a product of Eilenberg-Mac Lane spaces
$$\mprod_{1\leq q\leq \infty}K(H_{q}(X),q)$$

\item The natural map
$$SP(X) \longrightarrow SP(X/A)$$
is a quasi-fibration with quasi-fibre $SP(A)$.
\end{enumerate}
\end{thm}
\skp{0.2}
A natural map, 
$$\widehat{\bm{\theta}} \colon SP^{q_1}(X_{1})\wedge  SP^{q_2}(X_{2}) \wedge \cdots\wedge SP^{q_m}(X_{m}) \longrightarrow
SP^{q}(X_1 \wedge X_2 \wedge \cdots \wedge X_m),$$
for $q = q_{1}q_2 \cdots q_m$,\; is constructed next by setting
\begin{multline}
\widehat{\bm{\theta}}\big(\big[[x_{11},x_{12},\ldots,x_{1q_1}],[x_{21},x_{22},\ldots,x_{2q_2}],\ldots,
[x_{m1},x_{m2},\ldots,x_{mq_m}]\big]^{\wedge}\big)\\[3mm]
\;=\;\Big[\mprod_{\substack{1\leq j_t\leq q_t\\[0.7mm] 1\leq t \leq m}} 
\big[x_{1j_1},x_{2j_2},\ldots,x_{mj_m}\big]^{\wedge}\Big]\phantom{mmmmmmmmmmmm}
\end{multline}
\nd where here, square brackets $[\;]$ are used to denote equivalence classes in the 
symmetric product, and $[\;]^{\wedge}$ for the smash products. Next we introduce an 
extension of $\widehat{\bm{\theta}}$ which we shall use throughout to deduce  the main results.
\begin{thm}\label{thm:thetahat} 
The construction of the map \;$\widehat{\bm{\theta}}$ extends in a natural way to give
$$\widehat{\bm{\theta}} \colon SP(X_{1})\wedge  SP(X_{2}) \wedge \cdots\wedge  SP(X_{m}) \longrightarrow
SP(X_{1} \wedge X_{2} \wedge \cdots \wedge X_{m})$$
\nd a map of colimits.
\end{thm}

\begin{proof}
It suffices to check that the diagram below commutes
\skp{0.3}
\begin{tikzcd}
SP^{q_1}(X_1)\wedge \cdots \wedge SP^{q_k}(X_k) \wedge \cdots \wedge  SP^{q_m}(X_m)
\arrow[r, "\widehat{\bm{\theta}}"] \arrow[d, "\bm{e}"] 
&SP^{q}(X_1 \wedge X_2 \wedge \cdots \wedge X_m) \arrow[d, "\bm{e}"]\\
SP^{q_1}(X_1)\wedge \cdots \wedge SP^{q_k + 1}(X_k) \wedge \cdots \wedge  SP^{q_m}(X_m)
\arrow[r, "\widehat{\bm{\theta}}"]
&SP^{q'}(X_1 \wedge X_2 \wedge \cdots \wedge X_m)
\end{tikzcd}
\skp{0.3}
\nd where here, $q = q_{1}q_2 \cdots q_m$,\; $q' =  q_{1}q_2 \cdots q_{k-1}(q_k +1)q_{k+1} \cdots q_m$  and the map 
$\bm{e}$ is as in \eqref{eqn:mape}.
consider then,
\begin{multline*}\bm{e}\circ \bm{\theta}\big(\big[[x_{11},x_{12},\ldots,x_{1q_1}],[x_{21},x_{22},\ldots,x_{2q_2}],\ldots,
[x_{m1},x_{m2},\ldots,x_{mq_m}]\big]^{\wedge}\big)\\[2mm]
=\; \bm{e}\Big(\Big[\mprod_{\substack{1\leq j_t\leq q_t\\[0.7mm] 1\leq t \leq m}} 
\big[x_{1j_1},x_{2j_2},\ldots,x_{mj_m}\big]^{\wedge}\Big]\Big)
\;=\; \Big[\mprod_{\substack{1\leq j_t\leq q_t\\[0.7mm] 1\leq t \leq m}} 
\big[x_{1j_1},x_{2j_2},\ldots,x_{mj_m}\big]^{\wedge},
\;\underbrace{[\ast]^{\wedge},\ldots,[\ast]^{\wedge}}_{q'-q}\;\Big].
\end{multline*}
\skp{0.2}
\nd On the other hand,
\begin{align*}\phantom{m}\bm{\theta}\circ \bm{e} \big(\big[[&x_{11},x_{12},
\ldots,x_{1q_1}],[x_{21},x_{22},\ldots,x_{2q_2}],\ldots,
[x_{m1},x_{m2},\ldots,x_{mq_m}]\big]^{\wedge}\big)\\[2mm]
&\;=\; \bm{\theta}\big(\big[[x_{11},x_{12},\ldots,x_{1q_1}],\ldots,[x_{21},x_{22},\ldots,x_{2q_k,}\ast],\ldots,
[x_{m1},x_{m2},\ldots,x_{mq_m}]\big]^{\wedge}\big)\\[2mm]
&\;=\; \Big[\mprod_{\substack{1\leq j_t\leq q_t\\[0.7mm] 
1\leq t \leq m}} \big[x_{1j_1},x_{2j_2},\ldots,x_{mj_m}\big]^{\wedge}
 \mprod_{\substack{1\leq j_t\leq q_t\\[0.7mm] 1\leq t \leq m, \;t\neq k}} 
\big[x_{1j_1},x_{2j_2},\ldots,x_{mj_{m}}, \ast\big]^{\wedge}\Big]\\[2mm]
&\;=\; \Big[\mprod_{\substack{1\leq j_t\leq q_t\\[0.7mm] 1\leq t \leq m}} 
\big[x_{1j_1},x_{2j_2},\ldots,x_{mj_m}\big]^{\wedge},
\;\underbrace{\ast,\ldots,\ast}_{q'-q}\;\Big] \qedhere
\end{align*}
\end{proof}

\nd A simple example illustrates the proof of Theorem \ref{thm:thetahat} for $m = 2$, $q = 2$ and $q' =4$.
\begin{equation}\begin{tikzcd}
SP^{1}(X_1)\wedge SP^{2}(X_2)
\arrow[r, "\widehat{\bm{\theta}}"] \arrow[d, "\bm{e}\times 1"] 
&SP^{2}(X_1 \wedge X_2) \arrow[d, "\bm{e}"]\\
SP^{2}(X_1)\wedge SP^{2}(X_2)
\arrow[r, "\widehat{\bm{\theta}}"]
&SP^{4}(X_1  \wedge X_2)
\end{tikzcd}\end{equation}
\skp{0.3}
\nd Here,  
$$(\bm{e}\circ \widehat{\bm{\theta}})\big(\big[[x_{11}],[x_{21},x_{22}]\big]^{\wedge}\big) 
= e\big[[x_{11},x_{21}]^{\wedge},[x_{11},x_{22}]^{\wedge}\big] = 
\big[[x_{11},x_{21}]^{\wedge},[x_{11},x_{22}]^{\wedge}, [\ast]^{\wedge},[\ast]^{\wedge}\big],$$
\nd whereas,
\begin{align*}
\big(\widehat{\bm{\theta}}\circ (e\times 1)\big)\big(\big[[x_{11}],[x_{21},x_{22}]\big]^{\wedge}\big) &= 
\widehat{\bm{\theta}}\big(\big[[x_{11},\ast],[x_{21},x_{22}]\big]^{\wedge}\big)\\ &=
\big[[x_{11},x_{21}]^{\wedge},[x_{11},x_{22}]^{\wedge}, [\ast, x_{21}]^{\wedge},[\ast, x_{22}]^{\wedge}\big].
\end{align*}
\nd The diagram commutes because  both $[\ast, x_{21}]^{\wedge}$ and $[\ast, x_{22}]^{\wedge}$ 
\;equal\; $[\ast]^{\wedge}$. 

\section{Connections between symmetric products and polyhedral products}\label{sec;connections}
Consider a simplicial complex $K$ on $[m]$ and  a family of pointed CW pairs $(\underline{X}, \underline{A})$.
We adopt  the notation
\begin{equation}\label{eqn:star}
\big(\underline{SP^{\ast}(X)}, \underline{SP^{\ast}(A)}\big)  \;=\; 
\big\{\big(SP^{q_i}(X_i),SP^{q_i}(A_i)\big)\big\}_{i=1}^m,
\end{equation}
and construct a structure map
\begin{equation}\label{eqn:strmap}
\bm{\zeta}\colon \widehat{Z}\big(K;\big(\underline{SP(X)}, \underline{SP(A)}\big)\big) \longrightarrow 
SP\big(\widehat{Z}(K;(\underline{X}, \underline{A}))\big).
\end{equation}
by considering first the composite map
\begin{multline}\label{eqn:iota}
\phantom{m}\widehat{Z}\big(K;\big(\underline{SP^{\ast}(X)}, \underline{SP^{\ast}(A)}\big)\big)  
\xrightarrow{\phantom{m}\bm{\iota}\phantom{m}}  
SP^{q_1}(X_1)\wedge SP^{q_2}(X_2) \wedge\cdots \wedge  SP^{q_m}(X_m)\\
\hspace{0.3in}\xrightarrow{\phantom{m}\widehat{\bm{\theta}}\phantom{m}}   
SP^{q}(X_1 \wedge X_2 \wedge \cdots \wedge X_m)\phantom{mmmmmmmmm}
\end{multline}
\skp{0.2}

\begin{lem}\label{lem:zetadefn}
A map $\bm{\zeta}$ exists in the diagram below making the diagram commute
\begin{equation}\label{eqn:extzeta}
\hspace{0.0cm}\begin{tikzcd}[column sep = 2cm]
\widehat{Z}\big(K;\big(\underline{SP^{\ast}(X)}, \underline{SP^{\ast}(A)}\big)\big) 
\arrow[r, "\widehat{\Theta}\;\circ\; \bm{\iota}"] \arrow[dr, "\bm{\zeta} ", dashrightarrow]
&SP^{q}(X_1 \wedge X_2 \wedge \cdots \wedge X_m) \\ 
&SP^{q}\big(\widehat{Z}(K;(\underline{X},\underline{A}))\big) \arrow[u,"SP^{q}(\bm{\iota}) "]
\end{tikzcd}
\end{equation}
\nd where 
$\bm{\iota}\colon \widehat{Z}(K;(\underline{X},\underline{A})) \lhook\joinrel\longrightarrow 
X_1 \wedge X_2 \wedge \cdots \wedge X_{m}$\; is the inclusion.
Moreover, the map $\bm{\zeta}$ extends to a map at 
level of infinite symmetric products
\begin{equation}\label{eqn:zetainf}
\bm{\zeta}\colon \widehat{Z}\big(K;\big(\underline{SP(X)}, \underline{SP(A)}\big)\big) 
\longrightarrow
SP\big(\widehat{Z}(K;(\underline{X},\underline{A}))\big).
\end{equation}
\end{lem}
\skp{0.2}
\begin{proof}
We use the indexing from  \eqref{eqn:star}, and begin by defining
$$\bm{\zeta}\colon \widehat{D}_{(\underline{SP^{\ast}(X)}, \underline{SP^{\ast}(A)})}(\sigma)
\longrightarrow
SP^{q}\big(\widehat{D}_{(\underline{X}, \underline{A})}(\sigma)\big)$$
\nd for $\sigma \in K$, where
\begin{equation}\label{eqn:d.sp}
\widehat{D}_{(\underline{SP^{\ast}(X)}, \underline{SP^{\ast}(A)})}(\sigma) =\mwedge^m_{i=1}Y_i,\quad {\rm where}\quad
Y_i=\left\{\begin{array}{lcl}
SP^{q_i}(X_i)  &{\rm if} & i\in \sigma\\
SP^{q_i}(A_i) &{\rm if} & i\in [m]-\sigma,
\end{array}\right.
\end{equation}
\nd and
\begin{equation}\label{eqn:d.sigma.hat}
\widehat{D}_{(\underline{X}, \underline{A})}(\sigma) =\mwedge^m_{i=1}W_i,\quad {\rm where}\quad
W_i=\left\{\begin{array}{lcl}
X_i  &{\rm if} & i\in \sigma\\
A_i &{\rm if} & i\in [m]-\sigma,
\end{array}\right.
\end{equation}
\nd by
\begin{multline}\label{eqn:zetadefn}
\bm{\zeta}\big(\big[[x_{11},x_{12},\ldots,x_{1q_1}],[x_{21},x_{22},\ldots,x_{2q_2}],\ldots,
[x_{m1},x_{m2},\ldots,x_{mq_m}]\big]^{\wedge}\big)\\[3mm]
\;=\;\Big[\mprod_{\substack{1\leq j_t\leq q_t\\[0.7mm] 1\leq t \leq m}} 
\big[x_{1j_1},x_{2j_2},\ldots,x_{mj_m}\big]^{\wedge}\Big]\phantom{mmmmmmmmmmmm}
\end{multline}
The key point which makes the target of $\bm{\zeta}$ equal to
$SP^{q}\big(\widehat{D}_{(\underline{X}, \underline{A})}(\sigma)\big)$ is the observation that if
a point $[x_{r1},x_{r2},\ldots,x_{rq_r}]^{\wedge}$ is in $\widehat{A}_r^{q_r}$ then
$$\big[x_{1j_1},x_{2j_2},\ldots,x_{rj_r}, \ldots ,x_{mj_m}\big]^{\wedge}\; \in \;
X_1\wedge\cdots\wedge X_{r-1}\wedge A_r \wedge X_{r+1}\wedge\cdots\wedge X_m.$$

Next, we need to check that the diagram following commutes.
\begin{equation}\label{eqn:ddiagforzeta}
\begin{tikzcd}[column sep = 1.6cm, row sep = 0.8cm]
 \widehat{D}_{(\underline{SP^{\ast}(X)}, \underline{SP^{\ast}(A)})}(\sigma)
\arrow[r, "{\bm{\zeta}}"] \arrow[d, "\bm{\ell}"] 
&SP^{q}\big(\widehat{D}_{(\underline{X}, \underline{A})}(\sigma)\big) \arrow[d, "SP^{q}(\ell)"]\\
 \widehat{D}_{(\underline{SP^{\ast}(X)}, \underline{SP^{\ast}(A)})}(\tau)
\arrow[r, "{\bm{\zeta}}"] 
&SP^{q}\big(\widehat{D}_{(\underline{X}, \underline{A})}(\tau)\big)
\end{tikzcd}\end{equation}
\nd whenever $\sigma \xhookrightarrow{\phantom{n}\ell\phantom{n}} \tau$ in $K$, but this is immediate from
the definitions \eqref{eqn:d.sigma.hat} and \eqref{eqn:d.sp}. Taking colimits with respect to the diagram of $K$, we get the
dashed arrow of \eqref{eqn:extzeta},
\begin{equation}\label{eqn:finitezeta}
\bm{\zeta}\colon \widehat{Z}\big(K;\big(\underline{SP^{\ast}(X)}, \underline{SP^{\ast}(A)}\big)\big) 
\longrightarrow SP^{q}\big(\widehat{Z}(K;(\underline{X},\underline{A}))\big)
\end{equation}
It remains to check the commutativity of diagram \eqref{eqn:extzeta}. Let
$$\big[[x_{11},x_{12},\ldots,x_{1q_1}],[x_{21},x_{22},\ldots,x_{2q_2}],\ldots,
[x_{m1},x_{m2},\ldots,x_{mq_m}]\big]^{\wedge} \in D_{(\underline{SP^{\ast}(X)}, \underline{SP^{\ast}(A)})}(\sigma),$$
\nd we have,
\begin{align*}
\phantom{m}(\widehat{\bm{\theta}}\circ \bm{\iota})\big(\big[[&x_{11},x_{12},\ldots,x_{1q_1}],[x_{21},x_{22},\ldots,x_{2q_2}],\ldots,
[x_{m1},x_{m2},\ldots,x_{mq_m}]\big]^{\wedge}\big)\\[2mm]
&=\;\widehat{\bm{\theta}}\big(\big[[x_{11},x_{12},\ldots,x_{1q_1}],[x_{21},x_{22},\ldots,x_{2q_2}],\ldots,
[x_{m1},x_{m2},\ldots,x_{mq_m}]\big]^{\wedge}\big)\\[2mm]
&=\;\Big[\mprod_{\substack{1\leq j_t\leq q_t\\[0.7mm] 1\leq t \leq m}} 
\big[x_{1j_1},x_{2j_2},\ldots,x_{mj_m}\big]^{\wedge}\Big]
\end{align*}
Also,
\begin{align*}
\phantom{m}\big(SP^{q}(\bm{\iota})\circ \bm{\zeta}\big)\big(\big[[&x_{11},x_{12},\ldots,
x_{1q_1}],[x_{21},x_{22},\ldots,x_{2q_2}],\ldots,
[x_{m1},x_{m2},\ldots,x_{mq_m}]\big]^{\wedge}\big)\\[2mm]
&=\;SP^{q}(\bm{\iota})\Big(\Big[\mprod_{\substack{1\leq j_t\leq q_t\\[0.7mm] 1\leq t \leq m}} 
\big[x_{1j_1},x_{2j_2},\ldots,x_{mj_m}\big]^{\wedge}\Big]\Big)\\[2mm]
&=\;\Big[\mprod_{\substack{1\leq j_t\leq q_t\\[0.7mm] 1\leq t \leq m}} 
\big[x_{1j_1},x_{2j_2},\ldots,x_{mj_m}\big]^{\wedge}\Big]
\end{align*}
\skp{0.2}
Finally, we need to check that the map $\bm{\zeta}$ extends to a map of the colimits defining the infinite 
symmetric products, as in \eqref{eqn:zetainf}. To this end, we fix
$k \in [m]$ and modify \eqref{eqn:star} by setting
\begin{equation}\label{eqn:starplusone}
\big(\underline{SP^{\ast_k}(X)}, \underline{SP^{\ast_k}(A)}\big)  \;=\; 
\big\{\big(SP^{q'_i}(X_i),SP^{q'_i}(A_i)\big)\big\}_{i=1}^m,
\end{equation}
\nd where 
$$q'_i  = \begin{cases} q_i  \hspace{9.6mm}\rm{if}\hspace{2mm} i \neq k\\
                         q_k +1  \hspace{2mm}\rm{if}\hspace{2mm}  i =k \end{cases}$$ 
\skp{0.2} 
\nd and $q' = q'_{1}q'_{2}\cdots q'_{m}$.  It suffices now to check the commutativity of the diagram
      
\begin{equation}\begin{tikzcd}[column sep = 1.1cm, row sep = 0.8cm]
\widehat{Z}\big(K;\big(\underline{SP^{\ast}(X)}, \underline{SP^{\ast}(A)}\big)\big) 
\arrow[r, "{\bm{\zeta}}"] \arrow[d, "\widehat{Z}(K;\;\bm{e})"] 
&SP^{q}\big(\widehat{Z}(K;(\underline{X},\underline{A}))\big)  \arrow[d, "e"]\\
\widehat{Z}\big(K;\big(\underline{SP^{\ast_k}(X)}, \underline{SP^{\ast_k}(A)}\big)\big) 
\arrow[r, "{\bm{\zeta}}"] 
&SP^{q'}\big(\widehat{Z}(K;(\underline{X},\underline{A}))\big)
\end{tikzcd}\end{equation}           
\nd In the notation of \eqref{eqn:zetadefn}, we have
\begin{multline*}(\bm{e}\circ \bm{\zeta})\big(\big[[x_{11},x_{12},\ldots,x_{1q_1}],[x_{21},x_{22},\ldots,x_{2q_2}],\ldots,
[x_{m1},x_{m2},\ldots,x_{mq_m}]\big]^{\wedge}\big)\\[2mm]
=\; \bm{e}\Big(\Big[\mprod_{\substack{1\leq j_t\leq q_t\\[0.7mm] 1\leq t \leq m}} 
\big[x_{1j_1},x_{2j_2},\ldots,x_{mj_m}\big]^{\wedge}\Big]\Big)
\;=\; \Big[\mprod_{\substack{1\leq j_t\leq q_t\\[0.7mm] 1\leq t \leq m}} 
\big[x_{1j_1},x_{2j_2},\ldots,x_{mj_m}\big]^{\wedge},
\;\underbrace{\ast,\ldots,\ast}_{q'-q}\;\Big].
\end{multline*}

\nd On the other hand,
\begin{align*}\phantom{m}(\bm{\zeta}\circ\bm{e})\big(\big[[&x_{11},x_{12},
\ldots,x_{1q_1}],[x_{21},x_{22},\ldots,x_{2q_2}],\ldots,
[x_{m1},x_{m2},\ldots,x_{mq_m}]\big]^{\wedge}\big)\\[2mm]
&\;=\; \bm{\zeta}\big(\big[[x_{11},x_{12},\ldots,x_{1q_1}],\ldots,[x_{21},x_{22},\ldots,x_{2q_k,}\ast],\ldots,
[x_{m1},x_{m2},\ldots,x_{mq_m}]\big]^{\wedge}\big)\\[2mm]
&\;=\; \Big[\mprod_{\substack{1\leq j_t\leq q_t\\[0.7mm] 
1\leq t \leq m}} \big[x_{1j_1},x_{2j_2},\ldots,x_{mj_m}\big]^{\wedge}
 \mprod_{\substack{1\leq j_t\leq q_t\\[0.7mm] 1\leq t \leq m, \;t\neq k}} 
\big[x_{1j_1},x_{2j_2},\ldots,x_{mj_{m}}, \ast\big]^{\wedge}\Big]\\[2mm]
&\;=\; \Big[\mprod_{\substack{1\leq j_t\leq q_t\\[0.7mm] 1\leq t \leq m}} 
\big[x_{1j_1},x_{2j_2},\ldots,x_{mj_m}\big]^{\wedge},
\;\underbrace{\ast,\ldots,\ast}_{q'-q}\;\Big] \qedhere
\end{align*}
\end{proof}
The next construction is standard.

\begin{lem}\label{lem:multext}
A map of the form $\phi\colon X \longrightarrow SP^{k}(Y)$, which induces a map 
$$\phi\colon X \longrightarrow SP(Y)$$  
admits a canonical multiplicative extension ${\psi}\colon SP(X) \longrightarrow SP(Y)$. This extension is
the identity map if $X = Y$ and the map ${\phi}$ is the inclusion $E_X$.

\begin{proof}
 The map $\psi$ is defined by the map $\phi$ as a composite
\begin{equation}\label{eqn:multext}
SP^{q}(X)\; \xrightarrow{SP^{q}(\phi)}\; SP^{q}\big(SP^{k}(Y)\big)\; \xrightarrow{\eta} \; SP^{qk}(Y)
\end{equation}
\nd where the map $\eta$ is given by:
\begin{multline*}
\big[[y_{11},y_{12},\ldots,y_{1k}],[y_{21},y_{22},\ldots,y_{2k}],\;\ldots\;, [y_{q1},x_{q2},\ldots,y_{qk}]\big]\\
\;\xmapsto{\phantom{mmm}}\;
[y_{11},y_{12},\ldots,y_{1k}, y_{21},x_{22},\ldots,y_{2k},\;\ldots\;, y_{q1},y_{q2},\ldots,y_{qk}]
\end{multline*}
so, writing $\phi(x_i) = [\phi(x_i)_1,\phi(x_i)_2,\ldots,\phi(x_i)_k] \in SP^{k}(Y)$, we have
\begin{multline*}
\psi\big([x_1,x_2,\ldots,x_q]\big) \;= \;\eta\big(\big[\phi(x_1),\phi(x_2),\ldots,\phi(x_q)\big]\big)\\
=\; [\phi(x_1)_1,\phi(x_1)_2,\ldots,\phi(x_1)_k, \phi(x_2)_1,\phi(x_2)_2,\ldots,\phi(x_2)_k,
\;\ldots\;, \phi(x_q)_1,\phi(x_q)_2,\ldots,\phi(x_q)_k]
\end{multline*} 

The map $\psi$ fits into a commutative diagram as follows.
\begin{equation}\begin{tikzcd}[column sep = 1.1cm, row sep = 0.8cm]
SP^{q}(X)
\arrow[r, "{\psi}"] \arrow[d, "\bm{e}"] 
&SP^{qk}(Y)  \arrow[d, "\bm{e}"]\\
SP^{q+1}(X)
\arrow[r, "{\psi}"] 
&SP^{(q+1)k}(Y).
\end{tikzcd}\end{equation} 
\nd More specifically, let \;$[x_1,x_2,\ldots,x_q] \in SP^{q}(X)$, then

\begin{multline*}
(\bm{e}\circ \psi)\big([x_1,x_2,\ldots,x_q]\big)
\;=\;\bm{e}\big(\eta([\phi(x_1),\phi(x_2),\ldots,\phi(x_q)])\big)\\
\;=\;\big[\phi(x_1)_1,\phi(x_1)_2,\ldots,\phi(x_1)_k,
\quad\ldots\quad, \phi(x_q)_1,\phi(x_q)_2,\ldots,\phi(x_q)_k,\;\underbrace{\ast,\ldots,\ast}_{k}\big]
\end{multline*}
On the other hand,
\begin{align*}
\hspace{-1in}(\psi\circ \bm{e})\big([x_1,x_2,\ldots,x_q]\big)
&\;=\;\psi(\big([x_1,x_2,\ldots,x_q,\ast]\big)\\
&\;=\;\eta\big([\phi(x_1),\phi(x_2),\ldots,\phi(x_q),\phi(\ast)]\big)\\[-1.5mm]
&\;=\;[\phi(x_1),\phi(x_2),\ldots,\phi(x_q),\;\underbrace{\ast,\ldots,\ast}_{k}\;]
\end{align*}
$$\phantom{mmmm}=\;\big[\phi(x_1)_1,\phi(x_1)_2,\ldots,\phi(x_1)_k,
\quad\ldots\quad, \phi(x_q)_1,\phi(x_q)_2,\ldots,\phi(x_q)_k,\;\underbrace{\ast,\ldots,\ast}_{k}\big]$$

The second  statement  of the lemma follows from the definition of the map $E_X$.
Notice further that, for the map $E_X$ of \eqref{eqn:defofE}, we see that
\begin{equation}\label{eqn:resttoX}
(\psi\circ E_X)\colon X \longrightarrow SP(Y)
\end{equation}\label{eqn:psiagrees}
\nd coincides with $\phi\colon X \longrightarrow SP(Y)$. \qedhere
\end{proof}
\end{lem}
\skp{0.2}

Applying Lemma \ref{lem:multext} to the map $\bm{\zeta}$ of \eqref{eqn:zetainf}, we get its multiplicative
extension
\begin{equation}\label{eqn:bigpsi}
\psi\colon SP\big(\widehat{Z}\big(K;\big(\underline{SP(X)}, \underline{SP(A)}\big)\big) 
\longrightarrow
SP\big(\widehat{Z}(K;(\underline{X},\underline{A}))\big).
\end{equation}
More properties of the map $\psi$ are given next.
\skp{0.3}
\begin{lem}\label{lem:spandpp}
Let $K$ be a simlicial complex on $[m]$ and $(\underline{X}, \underline{A})$ a family of finite pointed CW pairs.
\skp{0.1}
\begin{enumerate}\itemsep3mm
\item The maps $(E_{X_i},E_{A_i}) \colon (X_i,A_i) \longrightarrow 
\big(SP(X_{i}), SP(A_{i}\big)$ induce a morphism of\\ polyhedral smash products
$$E^{K}_{(\underline{X}, \underline{A})} \colon\; \widehat{Z}(K; (\underline{X},\underline{A})\big) \longrightarrow
\widehat{Z}\big(K;(\underline{SP(X)},\underline{SP(A)})\big).$$
\item There is a commutative diagram
\begin{equation*}
\hspace{0.0cm}\begin{tikzcd}[column sep = 2.6cm, row sep = 1.1cm]
 \widehat{Z}(K;(\underline{X},\underline{A}))
\arrow[r, "E^{K}_{(\underline{X}, \underline{A})}"] \arrow[dr, "E_{ \widehat{Z}(K;(\underline{X},\underline{A})}"']
&\widehat{Z}\big(K;(\underline{SP(X)},\underline{SP(A)})\big)
\arrow[d, "\bm{\zeta} "] \\ 
&SP\big(\widehat{Z}\big(K;(\underline{X},\underline{A})\big)\big)
\end{tikzcd}
\end{equation*}
\item There is a strictly commutative multiplicative diagram
\begin{equation*}
\hspace{0.0cm}\begin{tikzcd}[column sep = 2.6cm, row sep = 1.1cm]
SP\big(\widehat{Z}(K;(\underline{X},\underline{A}))
\arrow[r, "SP\big(E^{K}_{(\underline{X}, \underline{A})}\big)"] \arrow[dr, "1"']
&SP\big(\widehat{Z}\big(K;(\underline{SP(X)},\underline{SP(A)}\big)
\arrow[d, "\psi "] \\ 
&SP\big(\widehat{Z}\big(K;(\underline{X},\underline{A})\big)\big) 
\end{tikzcd}
\end{equation*}
\nd where
\begin{equation*}
\hspace{0.0cm}\begin{tikzcd}[column sep = 2.6cm, row sep = 1.1cm]
SP\big(\widehat{Z}(K;(\underline{X},\underline{A}))
\arrow[r, "SP\big(E^{K}_{(\underline{X}, \underline{A})}\big)"] 
&SP\big(\widehat{Z}\big(K;(\underline{SP(X)},\underline{SP(A)}\big)
\end{tikzcd}
\end{equation*}
\nd is multiplicative. 
\end{enumerate}
\end{lem}
\begin{proof}
Part $1$ is a consequence of the functoriality of the polyhedral smash product and  part $2$ follows from the
definition of $\bm{\zeta}$ \eqref{eqn:zetadefn}. Next, applying \eqref{eqn:resttoX} from Lemma \ref{lem:multext},
we see that the map $\psi$, (\eqref{eqn:bigpsi}), of part $3$, restricted to $\widehat{Z}\big(K;(\underline{SP(X)},\underline{SP(A)})\big)$,
coincides with the map $\bm{\zeta}$. The diagram of part $3$ follows by applying Lemma \ref{lem:multext} to the 
maps $E^{K}_{(\underline{X}, \underline{A})}$ and $\bm{\zeta}$ in the diagram of part $2$. \qedhere
\end{proof}
\skp{0.2}
This lemma admits further extensions for subspaces of polyhedral smash products.
\begin{lem}
Let $K$ be as in Lemma \ref{lem:spandpp}, and $(\underline{X}, \underline{A})$ and $(\underline{U}, \underline{V})$
be families of pairs of pointed, finite connected CW complexes. Assume further that there are maps of pointed pairs
$$g_i \colon (U_i,V_i) \longrightarrow \big(SP(X_i), SP(A_i)\big).$$
\nd Then
\begin{enumerate}\itemsep5mm
\item There are induced maps
$$\underline{g}\colon \underline{U}^{[m]} \longrightarrow SP(X_{1})\wedge  SP(X_{2}) \wedge \cdots\wedge SP(X_{m})$$ 
and
$$\widehat{D}_{(\underline{U}, \underline{V})}(\omega) \xrightarrow{\phantom{m}\underline{g}_{\omega}\phantom{m}}
\widehat{D}_{(\underline{SP(X)}, \underline{SP(A)})}(\omega)$$
\nd for $\omega \in K$.
\item For $\sigma \subset \tau \in K$, there is a strictly commutative diagram, obtained by restriction, as follows:
\begin{equation}\label{eqn:diagramofds}
\begin{tikzcd}
\widehat{D}_{(\underline{U}, \underline{V})}(\sigma)
\arrow[r, "\lambda_{\sigma}"] \arrow[d, "\bm{\ell}"] 
&SP\big(\widehat{D}_{(\underline{X}, \underline{A})}(\sigma)\big) \arrow[d, "SP(\bm{\ell})"]\\
\widehat{D}_{(\underline{U}, \underline{V})}(\tau)
\arrow[r, "\lambda_{\tau}"]
&SP\big(\widehat{D}_{(\underline{X}, \underline{A})}(\tau)\big)
\end{tikzcd}\end{equation}
\skp{0.3}
\nd where $\bm{i}\colon \tau \hookrightarrow \sigma$ denotes the inclusion of faces and, for $\omega \in K$, 
the map $\lambda_{\omega}$
is the composite
\begin{equation}\label{eqn:lambdafact}
\widehat{D}_{(\underline{U}, \underline{V})}(\omega) \xrightarrow{\phantom{m}\underline{g}_{\omega}\phantom{m}}
\widehat{D}_{(\underline{SP(X)}, \underline{SP(A)})}(\omega)
\xrightarrow{\phantom{m}{\bm{\zeta}}\phantom{m}}
SP\big(\widehat{D}_{(\underline{X}, \underline{A})}(\omega)\big)
\end{equation}
\nd  and $\bm{\zeta}$ is constructed in Lemma \ref{lem:zetadefn}.
\end{enumerate}
\end{lem}

\begin{proof}
The map $\underline{g}$  arises from the $m$-fold smash product of maps $g_i$ in a natural way. The
commutativity of the diagram \eqref{eqn:diagramofds} is obtained by splitting it into
two diagrams corresponding to the factorization of the map $\lambda_w$ as \eqref{eqn:lambdafact}.
The commutativity of the first diagram, corresponding to $\underline{g}_{\omega}$, is straightforward and the
second, corresponding to the map $\bm{\zeta}$, is \eqref{eqn:ddiagforzeta}. \qedhere
\end{proof}

\section{An extension from wedge decomposable pairs to the general case}\label{sec:gen.case}
The purpose of this section is to begin the task of extending  Theorem \ref{thm:main}, the Cartan formula 
for wedge decomposable
pairs, to a homological Cartan formula for arbitrary pointed path-connected pairs of finite CW-complexes 
$(\underline{X},\underline{A})$.
\begin{thm} \label{thm:Cartan_and_general_homology} (Homological Cartan formula)
Let $K$ be an abstract simplicial complex with $m$ vertices. Assume
that $(\underline{X},\underline{A})$ are pointed, path-connected pairs of finite 
CW-complexes for all $i$. There exist spaces $B_j,C_j, E_j$, $1 \leq j \leq m$, which are finite wedges of spheres and mod-$n$ Moore spaces together with a homotopy equivalence $$SP(\widehat{Z}(K; (\underline{B \vee C},\underline{B\vee E})))\to SP(\widehat{Z}(K; (\underline{X},\underline{A}))).$$ Thus $H_{\ast}(\widehat{Z}(K; (\underline{X},\underline{A})))$ is isomorphic to $H_{\ast}(\widehat{Z}(K; (\underline{B \vee C},\underline{B\vee E})))$ over the integers. This allows for a description of  $H_{\ast}(\widehat{Z}(K; (\underline{X},\underline{A})))$ in terms of the decompositions of Theorem \ref{thm:main} and Corollary \ref{cor:wedge}.

\end{thm}

\begin{rem}
In a forthcoming paper, the authors use this and Corollary \ref{cor:wedge} to describe products in the cohomology 
of a polyhedral product.
\end{rem}

Preliminary results required for the proof of Theorem \ref{thm:Cartan_and_general_homology} will occupy the remainder of this section. We begin with a definition.
\begin{defin}\label{defin:maps between infinite symmetric products}
The pairs $(\underline{U},\underline{V})$ and $(\underline{X},\underline{A})$
are said to have {\em strongly isomorphic homology} provided 

\begin{enumerate}\itemsep2mm
\item there are isomorphisms of singular homology groups
$$\alpha_j: H_*(U_j) \to H_*(X_j),$$ and $$\beta_j: H_*(V_j) \to H_*(A_j),$$
\item there is a commutative diagram

$$\begin{CD}
 \bar{H}_i(V_j)  @>{\lambda{_j}_*}>>  \bar{H}_i(U_j)  \\
  @V{\beta_j}VV        @VV{\alpha_j}V        \\
  \bar{H}_i(A_j) @>{\iota_j}{_*}>>  \bar{H}_i(X_j),
\end{CD} $$
\skp{0.2}
\nd  where  $\lambda_j: V_j \subset U_j$, and $\iota_j: A_j \subset X_j$ are the natural inclusions, and
\item there is an induced morphism of exact sequences
for which all vertical arrows are isomorphisms:
$$\begin{CD}
0 @>{}>>  ker({\lambda_j}_*)  @>{}>>  \bar{H}_i(V_j ) @>{{\lambda_j}_*}>>  \bar{H}_i(U_j ) @>{}>> coker({\lambda_j}_*)@>{}>>  0  \\
  @V{}VV      @V{\beta_j}VV        @V{\beta_j}VV       @VV{\alpha_j}V     @VV{\bar{\alpha}_j}V    @V{}VV  \\
0@>{}>>  ker({\iota_j}_*) @>{}>>  \bar{H}_i(A_j) @>{{\iota_j}*}>>  \bar{H}_i(X_j) @>{}>> coker({\iota_j}_*) @>{}>>  0
\end{CD}$$
\skp{0.2}
\nd where $\bar{\alpha}_j$ is induced by $\alpha_j$.

\item The maps of pairs $(\alpha_j,\beta_j)\colon(H_*(U_j), H_*(V_j) ) \to (H_*(X_j), H_*(A_j) )$
which satisfy conditions $1-3$ are said to {\em induce a strong homology isomorphism}.
\end{enumerate}
\end{defin}
\begin{rem}
The feature of the pairs $(\underline{U},\underline{V})$ and $(\underline{X},\underline{A})$
having strongly isomorphic homology groups implies, via the Kunneth Theorem, that the spaces 
$\widehat{D}_{(\underline{U}, \underline{V})}(\sigma)$ and
$\widehat{D}_{(\underline{X}, \underline{A})}(\sigma)$ have isomorphic homology groups.
\end{rem}

\begin{lem} \label{lem: strongly isomorphic homology and homology of smash products}
Given pointed, path-connected pairs of finite CW-complexes
 $(\underline{X},\underline{A})$, and  $(\underline{U},\underline{V})$
 with strongly isomorphic homology groups, and let $\sigma \in K$ be any face of the simplicial complex $K$,
 then there is an isomorphism of singular homology groups
 $$\bar{H}_{\ast}\big(\widehat{D}_{(\underline{X}, \underline{A})}(\sigma)\big) \longrightarrow
\bar{H}_{\ast}\big(\widehat{D}_{(\underline{U}, \underline{V})}(\sigma)\big).$$
\end{lem}
 
\nd The rest of the section is devoted to showing that isomorphisms can be chosen to be suitably compatible 
to pass to isomorphisms on homology for the full polyhedral product. 

\begin{lem} \label{lem: existence of wedges with strongly isomorphic homology}
Given pointed, path-connected pairs of finite CW-complexes
$(\underline{X},\underline{A})$, there exist wedges of spheres, and mod-$p^r$ Moore spaces
$(\underline{ B \vee C}, \underline{B\vee E})$ together with isomorphisms of singular homology groups
$$\alpha_j \colon H_*(B_j\vee C_j) \to H_*(X_j)$$
and 
$$\beta_j \colon H_*(B_j \vee E_j) \to H_*(A_j),$$ 
\nd which give strong homology isomorphisms, and the pairs   
$(\underline{ B \vee C}, \underline{B\vee E})$  satisfy condition of wedge decomposability in 
Definition \ref{def:wedgedecomp} that the inclusion $E_j \to C_j$ is null-homotopic.
\end{lem}

\begin{proof}
The proof of this lemma follows from the fact that $X_j$, and $A_j$ are finite, path-connected CW-complexes, 
and so all homology groups as well as kernels and cokernels are finite direct sums of cyclic abelian groups. 
Thus $B_j, C_j, E_j$ may be chosen to be wedges of spheres, and  mod-$p^r$ Moore spaces.
Some details are given for completeness.

Consider a pair of pointed path-connected CW-complexes $(X,A)$ together with the induced map in homology
$H_*A \to H_*X.$  Then for any fixed $j \geq 1$, both $H_jA$ and $H_jX$ are finite direct sums of abelian groups.
Any such finitely generated abelian group is a direct sum of cyclic abelian groups either of the form $\Z$ or $\Z/p^r\Z$
for some choice of $n$. In the case $\Z$, then $H_j(S^j) = \mathbb{Z}$. In the case of $\Z/p^r\Z$, then the mod-$p^r$ 
Moore space given by $P^{j+1}(\mathbb{Z}/p^r\mathbb{Z})$ satisfies $H_j\big(P^{j+1}(\mathbb{Z}/p^r\Z)\big) = \Z/p^r\Z.$ 
At this level, it is direct to realize maps on homology. The hypotheses of {\it strongly isomorphic homology} 
gives the naturality required.
\end{proof}

Our next goal is to establish a standard property of the infinite symmetric product.

\begin{lem} \label{lem: infinite symmetric products}
Let $U$ and $X$ be finite, pointed,  path connected  CW-complexes and 
$$\alpha\colon \widetilde{H}_*(U)  \longrightarrow  \widetilde{H}_*(X)$$ 
a homomorphism in singular homology. Then there is a multiplicative map
$$g \colon SP(U) \longrightarrow SP(X)$$ 
which satisfies $\pi_{\ast}(g) = \alpha$, and is a homotopy equivalence if the map
$\alpha$ is an isomorphism.
\end{lem}
\begin{proof}
Recall that the reduced singular homology of $SP(X)$ for any pointed, path-connected CW complex $X$ 
is isomorphic to 
$$\moplus_{1 \leq m \leq \infty} \widetilde{H}_{\ast}\widehat{SP}^{(m)}(X)$$ 
where $\widehat{SP}^{(m)}(X)$ denotes the $m$-fold symmetric smash product given by
$$\widehat{SP}^{(m)}(X) =(X \wedge X \wedge \cdots \wedge X)\big/\Sigma_m.$$
This result appears in the thesis of R.~J.~Milgram, implicitly in \cite{doldthom} and in 
\mbox{\cite[Corollary $4.8$]{cmt}}. Also, N.~Steenrod proved, in unpublished notes, that there is a 
homotopy equivalence
$$  SP\big(SP(X)\big) \to SP\Big(\mvee_{1 \leq m < \infty} \widehat{SP}^{(m)}(X)\Big).$$ 
Since $X$ and $U$ are finite, path-connected, pointed CW complexes with isomorphic 
singular homology groups,  there is a homotopy equivalence 
$$\phi_n: SP(\Sigma^n(U)) \longrightarrow SP(\Sigma^n(X))$$ 
for every natural number $n$.  Next, consider the composite 
$$\begin{CD}
\Sigma^n(U)  &@> E >>&  SP(\Sigma^n(U)) &@> {\phi_n} >>& SP(\Sigma^n(X))\\
\end{CD}$$
and denote the canonical multiplicative extension,  (lemma \ref{lem:multext}), of $\phi_n\circ E$ by
$$\psi_n: SP(\Sigma^n(U)) \to SP(\Sigma^n(X)).$$ 
\begin{rem}
This map might not be homotopic to the map $\phi_n$.
\end{rem}
\skp{0.1}
The next step in the proof of Theorem \ref{thm:Cartan_and_general_homology} is a lemma
which allows a direct translation of algebraic properties concerning homology groups to geometric 
maps on the level of symmetric products of polyhedral products.  It follows from the  
Dold-Thom theorem that the map $\psi_n$ satisfies the formula 
$$ \psi_*(u) = \alpha(u) + \Delta_{u} \quad  \text{where} \quad \Delta_{u} \in
\moplus_{2 \leq m}H_{\ast}\big(\widehat{SP}^m(\Sigma^n(X)\big).$$
Here the class $\Delta_{u}$ projects to zero in $H_{\ast}(\Sigma^nX) = 
H_{\ast}(\widehat{SP^1}(\Sigma^nX))$ on identifying 
the reduced homology of a space $\Sigma^n(X)$ with that of $X$.  Thus the canonical multiplicative 
extension $\psi_n: SP(\Sigma^n(U)) \to SP(\Sigma^n(X))$  induces a surjection on homology. 

It follows that $\psi_n$ induces a homology isomorphism since the homology groups of 
source and target are finitely generated abelian groups in each degree
which are abstractly isomorphic where the map is a surjection.
\end{proof}
\newpage
\begin{lem} \label{lem: homological.general.exact,wedge.decomposable}
Suppose that $(\underline{U},\underline{V})$ and $(\underline{X},\underline{A})$ are pointed, connected, 
pairs of finite CW-complexes, with strongly isomorphic homology groups, then  the map of pairs 
$$g: \big(SP(U), SP(V)\big) \longrightarrow \big(SP(X), SP(A)\big)$$ 
arising from lemma \ref{lem: infinite symmetric products}, induces {\it strongly isomorphic homology groups}.
That is, the induced map on homology gives a commutative diagram 
$$\begin{CD}
  \bar{H}_i(SP(V_j) ) @>{\lambda_*}>>  \bar{H}_i(SP(U_j) )  \\
@V{g_*}VV       @VV{g_*}V      \\
\bar{H}_i(SP(A_j)) @>{\iota_*}>>  \bar{H}_i(SP(X_j)) 
\end{CD}$$
together with a second commutative diagram for which all vertical arrows are isomorphisms:
$$\begin{CD}
ker(\lambda_*)  @>{}>>  \bar{H}_i(SP(V_j) ) @>{\lambda_*}>>  \bar{H}_i(SP(U_j) ) @>{}>> coker(\lambda_*)  \\ 
@V{g_*}VV        @V{g_*}VV       @VV{g_*}V     @VV{\bar{g}_*}V \\
  ker(\iota_*) @>{}>>  \bar{H}_i(SP(A_j)) @>{\iota_*}>>  \bar{H}_i(SP(X_j)) @>{}>> coker(\iota_*) 
\end{CD}$$ 
where $\bar{g}_*$ is induced by $g_*$.
\end{lem}

The proof of this lemma follows from the fact that $X_j, A_j, U_j$, and $V_j$ are finite connected CW-complexes, 
and so all homology groups as well as kernels and cokernels are finite direct sums of cyclic abelian groups. 

Finally, we shall need a version of the Projection Lemma due to D.~Quillen.
\samepage{\begin{lem}\cite[Projection Lemma $1.6$]{zz}\;\label{lem: quasifibrations}
Let $\mathcal D$ and  $\mathcal E$
denote finite diagrams of finite CW complexes over the same finite category $\mathfrak{C}$
for which all inclusions in the intersection poset are closed cofibrations. Furthermore
assume that 
$$U = \mcup_{\alpha \in \mathfrak{C}} D_{\alpha}  \quad\text{and}\quad X = \mcup_{\alpha \in \mathcal C} E_{\alpha}$$ 
and that there is a map 
$$\mu\colon SP(U) \longrightarrow SP(X)$$
which restricts to homotopy equivalences on 
$$\mu|_{SP(D_{\alpha})}:SP(D_{\alpha}) \longrightarrow SP(E_{\alpha})$$ 
for  all $\alpha \in\mathfrak{C}$. Then $\mu$ is a homotopy equivalence.
\end{lem}}

\newpage
\section{The proof of Theorem \ref{thm:Cartan_and_general_homology} completed}\label{sec:gen.case.end}
The proof of Theorem \ref{thm:Cartan_and_general_homology} uses  Theorem \ref{thm:main} and the interplay between
the Dold-Thom construction and polyhedral products given in section \ref{sec;connections}. 

Let $K$ be an abstract simplicial complex with $m$ vertices. Assume
that $(\underline{X},\underline{A})$ are pointed, path-connected pairs of finite 
CW-complexes for all $i$. Then by  Lemma \ref{lem: existence of wedges with strongly 
isomorphic homology}  we have
wedges of spheres, and mod-$p^r$ Moore spaces
$(\underline{ B \vee C}, \underline{B\vee E})$ together with isomorphisms of singular homology groups
$$\alpha_j \colon H_*(B_j\vee C_j) \to H_*(X_j) \hspace{5mm}\text{and} \hspace{5mm}
\beta_j \colon H_*(B_j \vee E_j) \to H_*(A_j),$$ 
\nd which give strong homology isomorphisms and the pairs   
$(\underline{ B \vee C}, \underline{B\vee E})$  satisfy condition of wedge decomposability in 
Definition \ref{def:wedgedecomp} that the inclusion $E_j \to C_j$ is null-homotopic.
\skp{0.1}
Next, Lemma \ref{lem: homological.general.exact,wedge.decomposable}  gives a map of pointed pairs 
$$g\colon \big(\underline{SP(B\vee C)},\underline{SP(B\vee E)}\big) \longrightarrow (\underline{SP(X)},\underline{SP(A)})$$ 
which induces a strong isomorphism in homology. Applying the functor $\widehat{D}_{(-, -)}(\sigma)$ to this map, we get a 
morphism.
$$\widehat{D}(\sigma;g) \colon\widehat{D}_{(\underline{SP(B\vee C)}, \;\underline{SP(B\vee E)})}(\sigma)
\longrightarrow \widehat{D}_{(\underline{SP(X)}, \underline{SP(A)})}(\sigma)$$
and, for each $\tau \subset \sigma$, a commutative diagram

$$\begin{CD}
\widehat{D}_{(\underline{SP(B\vee C)}, \;\underline{SP(B\vee E)})}(\tau) @>{\widehat{D}(\tau;g)}> >  
\widehat{D}_{(\underline{SP(X)}, \underline{SP(A)})}(\tau)\\
  @VV{\beta}V        @V{\beta}VV       \\
\widehat{D}_{(\underline{SP(B\vee C)}, \;\underline{SP(B\vee E)})}(\sigma) @>{\widehat{D}(\sigma;g)}> >  
\widehat{D}_{(\underline{SP(X)}, \underline{SP(A)})}(\sigma)\\
\end{CD}$$
\skp{0.1}
\nd  where each horizontal arrow is a homotopy equivalence. 

Further there are induced morphisms
of commutative diagrams via the structure map  $\zeta$ of Lemma \ref{lem:zetadefn}.
$$\begin{CD}
\widehat{D}_{(\underline{SP(B\vee C)}, \;\underline{SP(B\vee E)})}(\sigma) @>>{}>   
\widehat{D}_{(\underline{SP(X)}, \underline{SP(A)})}(\sigma)\\
  @V{\zeta}VV        @V{\zeta}VV       \\
SP\big(\widehat{D}_{(\underline{B\vee C}, \;\underline{B\vee E})}(\sigma)\big) @>>{}>   
SP\big(\widehat{D}_{(\underline{X}, \;\underline{A})}(\sigma)\big)\\
\end{CD}$$
\skp{0.1}
\nd Here, the horizontal arrows are homotopy equivalences by the Dold-Thom theorem   and Lemma
\ref{lem: strongly isomorphic homology and homology of smash products}, Thus  there is 
a map 
\skp{-0.1}
$$SP\big(\mcup_{\sigma \in K}\widehat{D}_{\underline{B\vee C}, \;\underline{B\vee E})}(\sigma)\big) \longrightarrow  
SP\big(\mcup_{\sigma \in K} \widehat{D}_{(\underline{X}, \;\underline{A})}(\sigma)\big)$$
\skp{0.1}
which is a homotopy equivalence by Lemma \ref{lem: quasifibrations}.
Theorem \ref{thm:Cartan_and_general_homology} folllows from this. \hfill $\square$

\newpage
\section{The Hilbert-Poincar\'e series for $Z(K;(X,A))$}\label{sec:hpseries}
We begin by reviewing some of the elementary properties of Hilbert-Poincar\'e series. Assume now that homology is taken with
coefficients in a field $k$ and all spaces are pointed, path connected with the homotopy type of CW-complexes. 
The Hilbert-Poincar\'e series
$$P(X,t) = \sum_{n}\big({\rm dim}_{k}H_{n}(X;k)\big)t^n$$
\nd and the reduced Hilbert-Poincar\'e series
$$ \overline{P}(X,t) = -1 + P(X,t)$$
\nd satisfy the following properties.
\begin{enumerate}\itemsep2mm
\item $P(X,t)P(Y,t) = P(X\times Y),t)$, and
\item  $\overline{P}(X,t) \overline{P}(Y,t)  = P(X\wedge Y,t)$.
\end{enumerate}
\skp{0.2}
\nd For a pair $(X,A)$ satisfying the conditions of Theorem \ref{thm:Cartan_and_general_homology}, we have 
\begin{equation}
\overline{P}\big(\widehat{Z}(K;(\underline{X},\underline{A})),t\big) = \overline{P}\big(\widehat{Z}(K;(\underline{U},\underline{V})),t\big)
\end{equation}
where the pair $(\underline{U},\underline{V}) = \underline{B\vee C}, \;\underline{(B\vee E)})$  is the pair determined by
$(\underline{X},\underline{A})$ and given by Lemma \ref{lem: existence of wedges with strongly isomorphic homology}.
Next, Theorem \ref{thm:main} gives

\begin{equation}
\overline{P}\big(\widehat{Z}(K;(\underline{U},\underline{V})),t\big) =
\sum_{I\leq [m]} \Big[\overline{P}\big(\widehat{Z}(K_I;(\underline{C},\underline{E})_I),t\big)\big) 
\cdot\lprod_{j\in [m]-I}\overline{P}(B_{j},t)\Big]
\end{equation}
We apply now Corollary \ref{cor:wedge} to refine this further and obtain the next theorem.

\begin{thm}\label{thm:poincareseries}
The reduced Hilbert-Poincar\'e series for $\widehat{Z}(K;(\underline{U},\underline{V}))$, and hence for
$\widehat{Z}(K;(\underline{X},\underline{A}))$ is given as follows,
\begin{equation*}\label{eqn:hpszhat}
\overline{P}\big(\widehat{Z}(K;(\underline{U},\underline{V})),t\big) =
\sum_{I\leq [m]}\Big[\sum_{\sigma \in K_I}\big[t\cdot\overline{P}(|\text{lk}_{\sigma}(K_{I})|,t)\cdot
\overline{P}\big(\widehat{D}_{\underline{C},\underline{E}}^{I}(\sigma),t\big)\big]\cdot\lprod_{j\in [m]-I}\overline{P}(B_{j},t)\Big]
\end{equation*}
\nd where here we use the convention that $t\cdot \overline{P}(|\varnothing|,t) = 1$, and
$\overline{P}\big(\widehat{D}_{\underline{C},\underline{E}}^{I}(\sigma),t\big)$ 
can be read off from \eqref{eqn:ced.sigma}. 
\end{thm}

Finally, Theorem \ref{thm:bbcgsplitting} gives now the  
Hilbert-Poincar\'e series for $Z(L;(U,V))$ and for $Z(L;(X,A)$, by applying
\eqref{eqn:hpszhat} for each $K = L_J, \; J \subseteq [m]$. 

\newpage
\section{Applications}\label{sec:applications}

\begin{exm}\label{exm:cp4additive}
Consider the composite
 \begin{equation}\label{eqn:cp2cp3}
f\colon \mathbb{C}P^2 \hookrightarrow \mathbb{C}P^3 \to \mathbb{C}P^3/\mathbb{C}P^1.
 \end{equation}
\nd and denote the mapping cylinder of \eqref{eqn:cp2cp3} by $M_f$.  We consider
$\widetilde{H}^{*}\big(\widehat{Z}\big(K;(M_f, \mathbb{C}P^3)\big)\big)$ and 
 $\widetilde{H}^{*}\big(Z(K;(M_f, \mathbb{C}P^3)\big)\big)$,  for any
simplicial complex $K$ on vertices $[m]$ and describe the Poincar\'{e} series for the special case
\begin{equation}\label{eqn:specialk}
K = \big\{\{v_1\},,\{v_2\},\{v_3\},\{v_1,v_2\},\{v_1,v_3\}\big\}.
\end{equation}
For $(X,A) = (M_f, \mathbb{C}P^2)$, we have
\begin{equation}\label{eqn:specificuv}
(U,V) = \Big(S^4\vee S^{6},\; S^{4} \vee S^{2}\Big).
\end{equation}
so that  
$$B = S^4, \;C= S^6 \;\text{and}\; E = S^{2}$$
Theorem \ref{thm:Cartan_and_general_homology} gives now
$$\widetilde{H}^{*}\big(\widehat{Z}\big(K;(M_f, \mathbb{C}P^3)\big)\big) \cong 
\widetilde{H}^{*}\big(\widehat{Z}(K;(B\vee C,B\vee E))\big).$$
Applying Theorem \ref{thm:main}, we get
\begin{align*}\widehat{Z}(K;(B\vee C,B\vee E)) &\xrightarrow{\simeq} 
\mvee_{I\leq [m]}\widehat{Z}\big(K_{I};\big(S^6, \;S^{2}\big)\big)\wedge 
\widehat{Z}\big(K_{[m]- I};(S^4,\;S^4\big)\big)\\
&= \mvee_{I\leq [m]}\widehat{Z}\big(K_{I};\big(S^6,\;S^{2}\big)\big)
\wedge (S^4\big)^{{\wedge}|[m]-I||}
\end{align*}
\nd where the last term represents the $(|[m]-I)|$-fold smash product.
Finally, Corollary \ref{cor:wedge} determines  completely
$$\widehat{Z}\big(K_{I};\big(S^6, \;S^{2}\big)\big)$$ 
by enumerating all  the links $|lk_{\sigma}(K_I)|$.
\skp{0.1}
Next, we describe the  Poincar\'{e} series for $K$ as in \eqref{eqn:specialk}.
According to \eqref{eqn:specificuv} the cohomology of $(M_f, \mathbb{C}P^2)$ satisfies
\begin{equation}\label{eqn:example}
H^{\ast}(M_f) =\mathbb{Z}\{ b_4, c_6\}  \quad \text{and} \quad H^{\ast}(\mathbb{C}P^2))=\mathbb{Z}\{e_2, b_4\}
\end{equation}
\nd where the dimensions of the classes are given by the subscripts.  We denote the classes $\{e_2,b_4,c_6\}$ supported on the vertex $i$ by $\{e^i_2,b_4^i,c_6^i\}$ and illustrate computation using Theorem 
\ref{thm:poincareseries} by determining the summand corresponding to 
\begin{equation}\label{eqn:caseathand}
I=\{2,3\}\quad \rm{and}\quad \sigma = \varnothing.
\end{equation}
\begin{enumerate}[(i)]\itemsep4mm
\item In this case, we have
$$\widehat{D}_{\underline{C},\underline{E}}^{I}(\sigma)=\ E_2\wedge E_3 = S^2\wedge S^2  \quad{\rm{and}}
 \quad \widetilde{H}\big(\widehat{D}_{\underline{C},\underline{E}}^{I}(\sigma)\big) = k\{e_2^2\otimes e_2^3\}$$	
\nd and so we get	$\overline{P}\big(\widehat{D}_{\underline{C},\underline{E}}^{I}(\sigma),t)=t^4$.  
\newpage
\item Next, since $[m]-I=\{1\}$,  we have
$$\mprod_{j\in \{1\}}\overline{P}(B_{j},t) =  \overline{P}(B_{1},t)\;\Longrightarrow\; \overline{P}(b^1_4, t) = t^4.$$
\item Turning to the links, we have
$$|lk_{\varnothing}(K_I)| \;=\; |\{\{2\},\{3\}\}|=S^0$$
\nd so that $t\cdot\overline{P}\big(|lk_{\varnothing}(K_I),t|\big)\;=\; t$.
\end{enumerate}
Finally, for the case at hand \eqref{eqn:caseathand}, Theorem \ref{thm:poincareseries}  contributes $t^9$ to 
the Poincar\'{e} series for $H^\ast\big(\widehat{Z}(K; \xa)\big)$. 	\nd Continuing in this way, we arrive at the 
(reduced) Poincar\'{e} series
$$\overline{P}\big(H^*(\widehat{Z}(K; (M_f, \mathbb{C}P^2)),t\big) = t^9+t^{11}+3t^{12}+5t^{14}+2t^{16}.$$ 
\end{exm}

Theorem \ref{thm:main} applies particularly well in cases where spaces have unstable attaching maps.
\begin{exm}
The homotopy equivalence $S^{1}\wedge Y \simeq \Sigma(Y)$ implies homotopy equivalences
\begin{equation}\label{eqn:he1}
\Sigma^{mq}\big(\widehat{Z}(K;(\underline{X},\underline{A}))\big)
\longrightarrow \widehat{Z}\big(K;\big(\underline{\Sigma^{q}(X)},\underline{\Sigma^{q}(A)}\big)\big)
\end{equation}
\nd where as usual,  $m$  is the number of vertices of $K$. Recall now that 
$SO(3) \cong \mathbb{R}\rm{P}^{3}$ 
and consider the pair
$$(X,A) = \big(SO(3), \mathbb{R}\rm{P}^{2}\big),$$
\nd for which there is a well known homotopy equivalence of pairs,  \cite[Section 1]{mukai}, 
\begin{equation}\label{eqn:he2}
\big(\Sigma^{2}\big(SO(3)\big), \;\Sigma^{2}\big(\mathbb{R}\rm{P}^{2}\big)\big) \longrightarrow 
\big(\Sigma^{2}\big(\mathbb{R}\rm{P}^{2}\big)\vee \Sigma^{2}(S^{3}),\;
\Sigma^{2}\big(\mathbb{R}\rm{P}^{2}\big)\big),
\end{equation}
\nd which makes the pair $\big(SO(3), \mathbb{R}\rm{P}^{2}\big)$ {\em stably wedge decomposable\/}.
\nd Next, combining \eqref{eqn:he1} and \eqref{eqn:he2}, we get a homotopy equivalence
$$\Sigma^{2m}\big(\widehat{Z}(K;(SO(3),\mathbb{R}\rm{P}^{2}))\big)
\longrightarrow 
\widehat{Z}\big(K;\big(\Sigma^{2}(\mathbb{R}\rm{P}^{2})\vee \Sigma^{2}(S^{3}),\;
\Sigma^{2}(\mathbb{R}\rm{P}^{2})\big).$$
\nd Finally, Theorem \ref{thm:Cartan_and_general_homology} allows us to conclude that 
$\widehat{Z}(K;(SO(3),\mathbb{R}\rm{P}^{2}))\big)$,
and hence the polyhedral product $Z(K;(SO(3),\mathbb{R}\rm{P}^{2}))$,
is stably a wedge of smash products of $S^{3}$ and $\mathbb{R}\rm{P}^{2}$.
\end{exm}
\skp{0.1}
Similar splitting results exist for
the polyhedral product whenever the spaces $X$ and $A$ split after finitely many suspensions. In particular, the fact that
$\Omega^{2}S^{3}$ splits stably into a wedge of Brown--Gitler spectra implies that the polyhedral product 
$Z\big(K; (\Omega^{2}S^{3}, \ast)\big)$ splits  stably  into a wedge of smash products of Brown--Gitler spectra.

\newpage
\bibliographystyle{amsalpha}

\begin{thebibliography}{99}
\bibitem{bbcg1}  A.~Bahri, M.~Bendersky, F.~R.~Cohen, and S.~Gitler, {\em Decompositions of the 
polyhedral product functor with applications to moment-angle complexes and related spaces}, PNAS, July, 2009, 106:12241--12244.
\bibitem{bbcg2} A.~Bahri, M.~Bendersky, F.~R.~Cohen and S.~Gitler, {\em The polyhedral product
functor: a method  of computation for moment-angle complexes, arrangements and 
related spaces\/}. Advances in Mathematics, {\bf 225} (2010), 1634--1668.
\bibitem{bbcg3} A.~Bahri, M.~Bendersky, F.~Cohen and S.~Gitler, {\em Cup products
in generalized moment-angle complexes\/}. Mathematical Proceedings of the Cambridge
Philosophical Society, {\bf 153}, (2012), 457--469.
\bibitem{bbcg10} A.~Bahri, M.~Bendersky, F.~R.~Cohen and S.~Gitler, {\em A spectral sequence for polyhedral products\/}. 
Advances in Mathematics, {\bf 308\/}, (2017), 767--814
\bibitem{bbcgsurvey}  A.~Bahri, M.~Bendersky and F.R.~~Cohen, 
{\em Polyhedral products and features of their homotopy theory\/}. Handbook of Homotopy Theory, Haynes Miller (ed.)
Chapman and Hall/CRC, 2019.
 \bibitem{bbp} I.~Baskakov, V.~Buchstaber and T.~Panov, {\em Cellular cochain complexes and torus actions\/},
 Uspeckhi.~Mat.~Nauk, {\bf 59}, (2004), no.~3, 159--160 (russian); Russian Math.~Surveys {\bf 89}, (2004), no.~3,
 562--563 (English translation).
 \bibitem{bp3} V.~Buchstaber and T.~Panov, {\em Actions of tori, combinatorial topology and homological algebra}, 
Russian Math. Surveys, {\bf 55} (2000), 825--921
\bibitem{cai1} L.~Cai,  {\em On products in a real moment-angle manifold\/}, J.~Math.~Soc.~Japan, {\bf 69}(2), (2017), 
503--528.
\bibitem{caichoi} L.~Cai and S.~Choi, {\em Integral cohomology groups of real toric manifolds and small covers\/}.
Online at: https://arxiv.org/pdf/1604.06988.pdf
\bibitem{cmt} F.~R.~Cohen, J.~P.~May and L.~R.~Taylor, {\em Splitting of certain spaces $CX$}, 
Math. Proc. Camb. Phil. Soc., {\bf 84}, 465--496
\bibitem{doldthom} A.~Dold and R.~Thom, {\em Quasifaserungen und unendliche symmetrische Produkte\/}, 
Annals of Mathematics. Second Series, {\bf 67}, 239--281, doi:10.2307/1970005
\bibitem{franz}  M.~Franz, {\em On the integral cohomology of smooth toric varieties\/}.\\ Online at:
https://arxiv.org/format/math/0308253 
\bibitem{gtshifted} J.~Grbi\'c and S. ~Theriault, {\em The homotopy type of the complement of a coordinate subspace 
arrangement}, Topology {\bf 46}, (2007), 357--396.
\bibitem{gt} J.~Grbi\'c and S.~Theriault, {\em The homotopy type of the polyhedral product for shifted complexes\/},
Advances in Mathematics, {\bf 245}, (2013) 690-715. DOI: 10.1016/j.aim.2013.05.002
\bibitem{ikshifted} K.~Iriye and D.~Kishimoto, {\em  Decompositions of polyhedral products for shifted complexes\/}, Advances
in  Mathematics, {\bf 245} (2013), 716--736.
\bibitem{ik5} K.~Iriye and D.~Kishimoto, {\em Fat wedge filtrations and decomposition of polyhedral products\/},
To appear in Kyoto J. Math. (DOI:10.1215/21562261-2017-0038)
\bibitem{ldm1} S.~L\'opez de Medrano, {\em Topology of the intersection of quadrics in} 
$\mathbb{R}^n$, Algebraic Topology, Arcata California, (1986), Lecture Notes in Mathematics,
{\bf 1370}  Springer-Verlag, (1989),  280--292.
\bibitem{mukai} J.~Mukai, {\em Some homotopy groups of the double suspension of the real projective space
$\mathbb{R}{\rm P}^6$}, $10^{\text{\tiny th}}$  Brazilian Topology Meeting, (S\~{a}o Carlos, 1996). Mat. Contemp. {\bf 13} (1997), 
235--249. 
\bibitem{zheng} Q.~Zheng, {\em The homology coalgebra and cohomology algebra of generalized 
moment-angle complexes\/}, Journal of Pure and Applied Algebra, (2012).
\bibitem{zheng2} Q.~Zheng, {\em The cohomology algebra of polyhedral product spaces}, Journal of Pure and Applied Algebra,
{\bf 220},  (2016), 3752--3776 
\newpage
\bibitem{zz} G.~M.~Ziegler and R.~T.~\v{Z}ivaljevi\'c, {\em Homotopy types of subspace arrangements via diagrams 
of spaces},  Math.~Ann.~ {\bf 295}, 527--548 (1993). {Online at: \url{https://doi.org/10.1007/BF01444901}}
\end{thebibliography}

\end{document}